\documentclass[11pt]{amsart}

\usepackage{amsmath}
\usepackage{amsfonts}
\usepackage{amssymb}
\usepackage{amsthm}
\usepackage{mathptmx}
\usepackage{mathrsfs}
\usepackage{latexsym}
\usepackage{times}
\usepackage[mathscr]{euscript}
\usepackage[isolatin]{inputenc}

\usepackage[pagebackref]{hyperref}
\usepackage[backrefs,alphabetic,initials]{amsrefs}

\DeclareMathAlphabet{\mathpzc}{OT1}{pzc}{m}{it}

\newtheorem{thm}{Theorem}[section]
\newtheorem{lem}[thm]{Lemma}
\newtheorem{prop}[thm]{Proposition}
\newtheorem{cor}[thm]{Corollary}

\theoremstyle{definition}
\newtheorem{defn}[thm]{Definition}
\newtheorem{ex}[thm]{Example}

\theoremstyle{remark}
\newtheorem{rem}[thm]{Remark}
\usepackage{color}

\newcommand{\spk}{\mathfrak{sp}}

\newcommand{\ok}{\mathfrak{o}}

\newcommand{\zk}{\mathfrak{z}}

\newcommand{\g}{\mathfrak{g}}

\newcommand{\lk}{\mathfrak{l}}

\newcommand{\qk}{\mathfrak{q}}

\newcommand\CC{\mathbb C}

\newcommand\PP{\mathbb P}

\newcommand{\Zs}{\mathscr Z}

\newcommand{\ad}{\operatorname{ad}}

\newcommand{\im}{\operatorname{Im}}

\newcommand{\End}{\operatorname{End}}
\newcommand{\Der}{\operatorname{Der}}

\renewcommand\hat\widehat
\renewcommand\tilde\widetilde 
\newcommand{\spa}{\operatorname{span}}

\newcommand{\OO}{\operatorname{O}}

\newcommand{\oplusp}{{ \ \overset{\perp}{\mathop{\oplus}} \ }}

\newcommand{\ps}{\PP^1}

\begin{document}

\date{\today}

\title[Solvable quadratic Lie algebras in low dimensions]{Solvable quadratic Lie algebras in low dimensions}
\author{Tien Dat Pham, Anh Vu Le, Minh Thanh Duong}

\address{Department of Physics, Ho Chi Minh city University of Pedagogy, 280 An Duong Vuong, Ho Chi Minh city, Vietnam.}

\email{thanhdmi@hcmup.edu.vn}
\email{vula@uel.edu.vn}


\date{\today}
\begin{abstract}
  In this paper, we classify solvable quadratic Lie algebras up to dimension 6. In dimensions smaller than 6, we use the Witt decomposition given in \cite{Bou59} and a result in \cite{PU07} to obtain two non-Abelian indecomposable solvable quadratic Lie algebras. In the case of dimension 6, by applying the method of double extension given in \cite{Kac85} and \cite{MR85} and the classification result of singular quadratic Lie algebras in \cite{DPU}, we have three families of solvable quadratic Lie algebras which are indecomposable and not isomorphic.
\end{abstract}
\maketitle

\section{Introduction}

Throughout the paper, the base field is the complex field $\CC$. All considered vector spaces are finite-dimensional complex vector spaces.

As we known, the Killing form is a useful tool in studying semisimple Lie algebras. The Cartan criterion in the classification of Lie algebras or the proof of the Kostant-Morosov theorem, an important tool in the classification of adjoint orbits of classical Lie algebras $\ok(m)$ and $\spk(2n)$, base on the non-degeneracy of the Killing form. Another remarkable property of the Killing form is the invariance. Therefore, it is natural to give a question that are there Lie algebras not necessarily semisimple for which there exists a non-degenerate invariant bilinear form? We call them {\em quadratic} Lie algebras. The answer is confirmed. For instance, let $\g$ be a Lie algebra, $\g^*$ be its dual space and $\ad^*:\g\rightarrow \End(\g^*)$ denoted by the coadjoint representation of $\g$ in $\g^*$. Then the vector space $\bar{\g}=\g\oplus \g^*$ with the product defined by:
\[ [X+f, Y+g]_{\bar{\g}}=[X,Y]_{\g} + \ad^*(X)(g)-\ad^*(Y)(f)
\] 
becomes a quadratic Lie algebra with an invariant bilinear form given by: \[B(X+f, Y+g) = f(Y)+g(X)\] for all $X,\ Y \in \g,\ f,\ g\in\g^*$.

Quadratic Lie algebras are an interesting algebraic object which is in relation with many problems in Mathematics and Physics (see \cite{Bor97}, \cite{FS96} and their references). Also, this notion can be generalized for Lie superalgebras or can be similarly considered for other algebras (\cite{Bor97}, \cite{BB99}, \cite{AB10} or \cite{ZC07}). Many works have been devoted to develop tools for the study of quadratic Lie algebras (\cite{Kac85}, \cite{MR85}, \cite{Bor97} and \cite{PU07}) in which the notion of double extension is an effective method to construct the structure of a quadratic Lie algebra. It is initiated by V. Kac in the solvable case \cite{Kac85} and after that developed by A. Medina and P. Revoy in the general case \cite{MR85}.

In this paper, we approach to quadratic Lie algebras in a familiar way that is in low dimensions. We focus on the solvable case and on the classification up to dimension 6. The classification of solvable quadratic Lie algebras up to dimension 4 can be found in \cite{ZC07} but here we redo it by a shorter way which is based on the Witt decomposition in \cite{Bou59} combined with a result in \cite{PU07}. We use this method to classify solvable quadratic Lie algebras of dimension 5 which have been done in \cite{DPU} by the method of double extension. In the case of dimension 6, we apply the classification of $\OO(n)$-adjoint orbits of the Lie algebra $\ok(n)$ (see \cite{CM93}) and the classification of singular quadratic Lie algebras give in \cite{DPU} to obtain three families of solvable quadratic Lie algebras which are indecomposable and not isomorphic.

The paper is organized in 3 sections. The first one is devoted to introduce basic definitions and preliminaries. The classification of solvable quadratic Lie algebras up to dimension 5 is shown in Section 2. In Section 3, we consider a particular case of the notion of double extension and apply it to obtain all indecomposable solvable quadratic Lie algebras of dimension 6. 

\section{Preliminaries}

\begin{defn}Let $\g$ be a Lie algebra. A bilinear form $B:\ \g\times\g\rightarrow \CC$ is called:

\begin{enumerate}
	\item[(i)] {\em symmetric} if  $B(X,Y) = B(Y,X)$ for all $X, \ Y \in\g$,
	\item[(ii)] {\em non-degenerate} if $B(X,Y) = 0$ for all $Y\in\g$ implies $X=0$,
	\item[(iii)] {\em invariant} if $B([X,Y],Z) = B(X,[Y,Z])$ for all $X,\ Y,\ Z \in\g$.
\end{enumerate}

A Lie algebra $\g$ is called {\em quadratic} if there exists a bilinear form $B$ on $\g$ such that $B$ is symmetric, non-degenerate and invariant.
\end{defn}

Let $(\g,B)$ be a quadratic Lie algebra and $V$ be a subspace of $\g$. Denote the {\em orthogonal component} of $V$ by $V^\bot=\{X\in\g\ |\ B(X,Y)=0,\ \forall \ Y\in V\}$ then we have:
\[ \dim(V)+\dim(V^\bot) = \dim(\g).
\]

An element $X$ in $\g$ is called {\em isotropic} if $B(X,X)=0$ and a subspace $W$ of $\g$ is called {\em totally isotropic} if $B(X,Y)=0$ holds all $X,\ Y\in W$. In this case, it is obvious that $W\subset W^\bot$.

The study of quadratic Lie algebras can be reduced to the study of {\em indecomposable} ones by the following decomposition \cite{Bor97}.

\begin{prop}\label{prop1.2}

Let $(\g,B)$ be a quadratic Lie algebra and $I$ be an ideal of $\g$. Then $I^\bot$ is also an ideal of $\g$. Moreover, if the restriction of $B$ on $I\times I$ is non-degenerate then the restriction of $B$ on $I^\bot\times I^\bot$ is also non-degenerate, $[I,I^\bot]=\{0\}$ and $I\cap I^\bot=\{0\}$.

\end{prop}

If the restriction of $B$ on $I\times I$ is non-degenerate then $I$ is called a {\em non-degenerate} ideal of $\g$. In this case, $\g=I\oplus I^\bot$. Since the direct product is the orthogonal direct product so for convenience, we use the notion $\g=I\oplusp I^\bot$.

\begin{defn}
We say a quadratic Lie algebra $\g$ {\em indecomposable} if $\g = \g_1
  \oplusp \g_2$, with $\g_1$ and $\g_2$ ideals of $\g$, implies $\g_1$
  or $\g_2 = \{0\}$.
\end{defn}

\begin{defn} Let $\g$ and $\g'$ be two Lie algebras endowed with non-degenerate invariant bilinear forms $B$ and $B'$ respectively. If $A$ is a Lie algebra isomorphism from $\g$ onto $\g'$ satisfying $B'(A(X),A(Y)) = B(X,Y)$ for all $X,Y\in\g$ then we say that $\g$ and $\g'$ are \emph{i-isomorphic} and $A$ is an \emph{i-isomorphism} from $\g$ onto $\g'$.
\end{defn}

Remark that the isomorphic and i-isomorphic notions may be not equivalent. An example can be found in \cite{DPU}.

Next, we introduce another decomposition that is called the {\em reduced} decomposition. This notion allows us to only focus on quadratic Lie algebras having a totally isotropic center.

\begin{prop} \label{prop2.8} \cite{PU07} \hfill

  Let $(\g,B)$ be a non-Abelian quadratic Lie algebra. Then there
  exist a central ideal $\zk$ and an ideal $\lk \neq \{0\}$ such
  that:

\begin{enumerate}

\item[(i)] $\g = \zk \oplusp \lk$, where $\left( \zk, B|_{\zk \times \zk} \right)$ and $\left(\lk,
    B|_{\lk \times \lk} \right)$ are quadratic Lie algebras. Moreover,
  $\lk$ is non-Abelian.

\item[(ii)] The center $\Zs(\lk)$ is totally isotropic, i.e. $\Zs(\lk)
  \subset [\lk, \lk]$ and
	\[\dim(\Zs(\lk))\leq \frac{1}{2} \dim(\lk)\leq\dim([\lk, \lk]).
	\]

\item[(iii)] Let $\g'$ be a quadratic Lie algebra and $A : \g \to \g'$ be a
  Lie algebra isomorphism. Then \[ \g' = \zk' \oplusp \lk'\] where
  $\zk' = A(\zk)$ is central, $\lk' = A(\zk)^\perp$, $\Zs(\lk')$ is
  totally isotropic and $\lk$ and $\lk'$ are isomorphic. Moreover if
  $A$ is an i-isomorphism then $\lk$ and $\lk'$ are i-isomorphic.

\end{enumerate}

\end{prop}

\begin{defn} A quadratic Lie algebra $\g\neq\{0\}$ is called {\em reduced} if $\Zs(\g)$ is totally isotropic.
\end{defn}

Note that if a quadratic Lie algebra $\g$ of dimension more than 1 is not reduced then there is a central element $X$ such that $B(X,X)\neq 0$. That means the ideal spanned by $X$ is non-degenerate. By Proposition \ref{prop1.2}, $\g$ is decomposable.
\section{Solvable quadratic Lie algebras up to dimension 5} \label{Section2}

In this section, we classify all indecomposable solvable quadratic Lie algebras up to dimension 5. The classification is based on the {\em Witt decomposition} given in \cite{Bou59} as follows:
\begin{prop}
Let $V$ be a finite-dimensional complex vector space endowed with a non-degenerate bilinear form $B$. Assume $U$ a totally isotropic subspace of $V$. Then there exist a totally isotropic subspace $W$ and a non-degenerate subspace $F$ of $V$ such that $\dim(U)=\dim(W)$, $F=(U\oplus W)^\bot$ and 
\[V=F\oplusp (U\oplus W)
\]
\end{prop}

As a consequence, if $\{X_1,...,X_n\}$ is a basis of $U$ then there exists a basis $\{Y_1,...,Y_n\}$ of $W$ such that $B(X_i,Y_j)=\delta_{ij}$ for all $1\leq i,j\leq n$.

We recall the classification of solvable quadratic Lie algebras up to dimension 4 in \cite{ZC07} but here there is a shorter proof based on a remarkable result in \cite{PU07} that if $\g$ is a non-Abelian quadratic Lie algebra then $\dim([\g,\g])\geq 3$.
\begin{prop}\label{prop3.2}
Let $\g$ be a solvable quadratic Lie algebra with $\dim(\g)\leq 4$. Then we have the following cases:
\begin{enumerate}
	\item[(i)] If $\dim(\g)\leq 3$ then $\g$ is Abelian.
	\item[(ii)] If $\dim(\g)=4$ and $\g$ is non-Abelian then $\g$ is i-isomorphic to the diamond Lie algebra $\g_4 = \spa\{X,P,Q,Z\}$, where the subspaces spanned by $\{X,P\}$ and $\{Q,Z\}$ are totally isotropic, $B(X,Z) = B(P,Q) = 1$, $B(X,Q)= B(P,Z) = 0$ and the Lie bracket is defined by $[X,P] = P$, $[X,Q] = -Q$, $[P,Q] = Z$, otherwise trivial.
\end{enumerate}
\end{prop}
\begin{proof}
Since $\g$ is solvable then $[\g,\g]\neq \g$. Therefore if $\dim(\g)\leq 3$ then $\g$ is Abelian. If $\dim(\g)=4$ and $\g$ is non-Abelian, we show that $\g$ is reduced. Indeed, if $\g$ is not reduced then there exists a central element $X$ such that $B(X,X)\neq 0$. That implies $I=\CC X$ a non-degenerate ideal of $\g$ and so we have
\[\g = I\oplusp I^\bot.
\]
Note that $I^\bot$ is a solvable quadratic Lie algebra of dimension 3 then $I^\bot$ is Abelian. Therefore, $\g$ is Abelian. This contradiction shows $\g$ reduced.

Since $[\g,\g]\neq \g$ and $\dim([\g,\g])\geq 3$ then $\dim([\g,\g])= 3$ and so $\dim(\Zs(\g))= 1$. Assume $\Zs(\g)=\CC Z$. Since $\Zs(\g)$ is totally isotropic then by the Witt decomposition there exists a one-dimensional totally isotropic subspace $W$ and a two-dimensional non-degenerate subspace $F$ of $\g$ such that:
\[\g=F\oplusp(\Zs(\g)\oplus W).
\]
Moreover, we can choose an element $X$ in $W$ and a basis $\{P,Q\}$ of $F$ such that $B(X,Z)=B(P,Q)=1$ and $B(P,P)=B(Q,Q)=0$.

By $[\g,\g]=\Zs(\g)^\bot$ then $[\g,\g]=\spa\{Z,P,Q\}$ and therefore we can assume
\[[X,P]=a_1Z+b_1P+c_1Q,\ [X,Q]=a_2Z+b_2P+c_2Q \ \text{and}\ [P,Q]=a_3Z+b_3P+c_3Q
\]
where $a_i,\ b_i, c_i\in\CC,\ 1\leq i\leq 3$.

Since $B(X,[X,P])=B([X,X],P)=0$ and $B(P,[X,P])=-B([P,P],X)=0$ one has $a_1=c_1=0$. Similarly, $a_2=b_2=0$ and $b_3=c_3=0$. Moreover, $B(X,[P,Q])=B([X,P],Q)=-B([X,Q],P)$ implies $b_1 = -c_2=a_3$.

By replacing $Z:=a_3Z$ and $X:=\frac{X}{a_3}$ we obtain the Lie bracket $[X,P]=P$, $[X,Q]=-Q$ and $[P,Q]=Z$. 
\end{proof}

Now, we continue with $\g$ an indecomposable solvable quadratic Lie algebra of dimension 5. It is obvious that $\g$ is reduced. By Proposition \ref{prop2.8}, there are only two cases: $\dim(\Zs(\g))= 1$ and $\dim(\Zs(\g))= 2$.
\begin{enumerate}
	\item[(i)] If $\dim(\Zs(\g))= 1$, assume $\Zs(\g)=\CC Z$. Then there exist an isotropic element $Y$ and a subspace $F$ of $\g$ such that $B(Z,Y)=1$ and $\g=F\oplusp(\CC Z\oplus \CC Y)$, where $F=(\CC Z\oplus \CC Y)^\bot$. We can choose a basis $\{P,Q,X\}$ of $F$ satisfying $B(P,X)=B(Q,Q)=1$, the other are zero.
	
	Since $[\g,\g]=\Zs(\g)^\bot$ then $[\g,\g]=\spa\{Z,P,Q,X\}$. Moreover, since $B([Y,X],Y)=B([Y,X],X) = 0$ then we can assume $[Y,X]=a_1X + b_1Q$ with $a_1,\ b_1\in\CC$. Similarly, we have the Lie bracket defined by:
 \[ [Y,Q]=a_2X+b_2P,\ \ [Y,P]=a_3Q+b_3P,\ \ [X,Q]=a_4X+b_4Z\]
\[ [X,P]=a_5Q+b_5Z \text{ and }\ [Q,P]=a_6P+b_6Z\]
where $a_i,\ b_i\in\CC$, $2\leq i\leq 6$.
	
	By the invariance of $B$, we obtain $a_1=b_5=-b_3$, $b_1 = b_4=-b_2$, $a_2 = b_6=-a_3$ and $a_4=a_6=-a_5$. Therefore, we rewrite Lie brackets as follows:
	\[ [Y,X]=xX+yQ,\ \ [Y,Q]=zX-yP,\ \ [Y,P]=-zQ-xP,\]
	\[ [X,Q]=wX+yZ,\ \ [X,P]=-wQ+xZ\ \ \text{ and }\ [Q,P]=wP+zZ\]
	where $x,\ y,\ z, \ w \in\CC$.
	
	If $w\neq 0$, set $A:=[X,Q]$, $B:=[X,P]$ and $C:=[Q,P]$ then one has $[A,B]=-w^2 A$, $[B,C]=-w^2 C$ and $[C,A]=-w^2 B$. It implies that the vector space $U=\spa\{A,B,C\}$ becomes a subalgebra of $\g$ and this subalgebra is not solvable. This is a contradiction by $\g$ solvable. So $w = 0$. In this case, it is easy to check that $zX-xQ+yP\in \Zs(\g)$. On the other hand, the numbers $x,\ y,\ z$ are not zero at the same time. So $\dim(\Zs(\g))> 1$. This is a contradiction then it does not happen this case.
		
	\item[(ii)] If $\dim(\Zs(\g))= 2$, assume  $\Zs(\g)=\spa\{Z_1,Z_2\}$. By the Witt decomposition, there exist elements $X_1,X_2$ and $T$ satisfying: $\g$ is spanned by $\{Z_1,Z_2,T,\linebreak X_1,X_2\}$, the subspace $W=\spa\{X_1,X_2\}$ is totally isotropic, the bilinear form $B$ is defined by $B(T,T)=1$, $B(X_i,Z_j)=\delta_{ij}$, $1\leq i,j\leq 2$, the other are zero.
	
	Since $[\g,\g]=\Zs(\g)^\bot$ then $[\g,\g]=\spa\{Z_1,Z_2,T\}$. Moreover, $B$ is invariant then we can assume $[X_1,X_2]=x T$, $[X_1,T]=y Z_2$ and $[X_2,T]=z Z_1$, where $x,\ y,\ z\in\CC$. By the invariance of $B$ again, we obtain $x=-y=z$. Replace $X_1:=\frac{X_1}{x}$ and $Z_1:=x Z_1$ to obtain $[X_1,X_2]=T$, $[X_1,T]=- Z_2$ and $[X_2,T]=Z_1$.		
\end{enumerate}

Finally, we have a classification of solvable quadratic Lie algebras of dimension 5 as follows.
\begin{prop}
Let $(\g,B)$ be an indecomposable solvable quadratic Lie algebra of dimension 5. Then there exists a basis $\{Z_1,Z_2,T,X_1,X_2\}$ of $\g$ such that $B(T,T)=1$, $B(X_i,Z_j)=\delta_{ij}$, $1\leq i,j\leq 2$, the other are zero and the Lie bracket is defined by $[X_1,X_2]=T$, $[X_1,T]=- Z_2$ and $[X_2,T]=Z_1$, the other are trivial.
\end{prop}
\section{Solvable quadratic Lie algebras of dimension 6}
\subsection{Double extension}
\begin{defn}
Let $(\g,B)$ be a quadratic Lie algebra. A derivation $D$ of $\g$ is called {\em skew-symmetric} if $D$ satisfies
\[ B(D(X),Y)=-B(X,D(Y)), \ \ \ \forall X,\ Y\in\g.
\]
\end{defn}

Denote by $\Der_a(\g)$ the vector space of skew-symmetric derivations of $\g$. By the invariance of $B$, all inner derivations of $\g$ are in $\Der_a(\g)$.
\begin{defn} (\cite{Kac85} and \cite{MR85})\hfill

Let $(\g,B)$ be a quadratic Lie algebra and $C\in \Der_a(\g)$. On the vector space
\[\bar{\g}=\g\oplus \CC e\oplus \CC f,
\]
we define the product:
\[[X,Y]_{\bar{\g}} = [X,Y]_{\g} +B(C(X),Y)f,\ \ [e,X]=C(X)\ \text{ and } [f,\bar{\g}]=0
\]
for all $X,\ Y\in\g$.

Then $\bar{\g}$ becomes a Lie algebra. Moreover, $\bar{\g}$ is a quadratic Lie algebra with an invariant bilinear form $\bar{B}$ defined by:
\[\bar{B}(e,e)=\bar{B}(f,f)=\bar{B}(e,\g)=\bar{B}(f,\g)=0, \ \bar{B}(X,Y)=B(X,Y) \text{ and } \bar{B}(e,f)=1
\]
for all $X,\ Y\in\g$. In this case, we call $\bar{\g}$ the {\em double extension of $\g$ by $C$}.
\end{defn}

The Lie algebra $\bar{\g}$ is also called the {\em double extension of $\g$ by the one-dimensional Lie algebra by means of $C$} (or a {\em one-dimensional double extension}, for short). The more general definition can be found in \cite{MR85}. However, the one-dimensional double extension is sufficient for studying solvable quadratic Lie algebras by the following proposition (see \cite{Kac85} or \cite{FS87}).
\begin{prop}\label{prop3.4}
Let $(\g,B)$ be a solvable quadratic Lie algebra of dimension $n$, $n\geq 2$.  Assume $\g$ non-Abelian. Then $\g$ is a one-dimensional double extension of a solvable quadratic Lie algebra of dimension $n-2$.
\end{prop}
\begin{rem} 
A particular case of one-dimensional double extensions is $\g$ Abelian, then $C$ is a skew-symmetric map in the Lie algebra $\ok(\g)$ and the Lie bracket on $\bar{\g}$ is given by:
\[[e,X]=C(X)\ \ \text{and}\ \ [X,Y]=B(C(X),Y)f,\ \ \ \forall\ X,\ Y\in\g.
\]
Such Lie algebras have been classified up to isomorphism and up to i-isomorphism in \cite{DPU}.
\end{rem}
\begin{ex}
Let $\qk$ be a two-dimensional complex vector space endowed with a non-degenerate symmetric bilinear form $B$. In this case, $\qk$ is called a two-dimensional {\em quadratic} vector space. We can choose a basis $\{P,Q\}$ of $\qk$ such that $B(P,P)=B(Q,Q)=0$ and $B(P,Q)=1$ (we call $\{P,Q\}$ a {\em canonical} basis of $\qk$). Let $C:\qk\rightarrow\qk$ be a skew-symmetric map of $\qk$, that is an endomorphism satisfying $B(C(X),Y)=-B(X,C(Y)$ for all $X,\ Y\in\qk$. In this example, we choose $C$ having the matrix with respect to the basis $\{P,Q\}$ as follows:

\[C = \begin{pmatrix} 1 & 0 \\ 0 & -1\end{pmatrix}.
\] 
Set the vector space
\[\g=\qk\oplus \CC X\oplus \CC Z
\]
and define the product:
\[[X,P]=C(P)=P,\ \ [X,Q]=C(Q)=-Q \ \ \text{ and }\ \ [P,Q]=B(C(P),Q)Z = Z
\]
Then $\g$ is the diamond Lie algebra given in Proposition \ref{prop3.2} (ii).
\end{ex}
\subsection{Solvable quadratic Lie algebras of dimension 6}\hfill

By Proposition \ref{prop3.4}, the key of the classification of solvable quadratic Lie algebras is describing skew-symmetric derivations of a solvable quadratic Lie algebra. More particular, in the case of dimension 6, it is necessary to describe skew-symmetric derivations of solvable quadratic Lie algebras of dimension 4. By the classification result in Proposition \ref{prop3.2}, we focus on the Abelian Lie algebra of dimension 4 and the diamond Lie algebra. We need the following lemma.

\begin{lem}\label{lem3.8}
Any skew-symmetric derivation of the diamond Lie algebra $\g_4$ is inner.
\end{lem}
\begin{proof}
Assume $\g=\spa\{X,P,Q,Z\}$ where $[X,P]=P$, $[X,Q]=-Q$ and $[P,Q]=Z$, the bilinear form $B$ is given by $B(X,Z) = B(P,Q)=1$, the other are zero. Let $D$ be a skew-symmetric derivation of $\g$. Since $\Zs(\g)$ is stable by $D$ then we can assume $D(Z)=xZ$ with $x\in\CC$. Moreover $D$ is skew-symmetric then
\[B(D(X),Z)=-B(X,D(Z)) = B(X,xZ)=-x.\]
So we can assume $D(X)=-xX+yP+zQ+wZ$ with $y,\ z,\ w\in\CC$. The ideal $[\g,\g]$ is also stable by $D$ then we write:
\[ D(P)=aP+bQ+cZ\ \ \text{ and }\ \ D(Q)=a'P+b'Q+c'Z
\]
where $a,\ b,\ c,\ a',\ b',\ c'\in \CC$.

One has $D(P)=D([X,P]) = [D(X),P] + [X,D(P)]$. Then we obtain $x=b=0$ and $z=-c$. By a straightforward calculating, we obtain $y=-c'$, $a'=w=0$, $a=-b'$ and then the matrix of $D$ corresponding to the basis $\{X,P,Q,Z\}$ is:

\[D = \begin{pmatrix} 0 & 0 & 0 & 0 \\ y & a & 0 & 0 \\ z &
  0 & -a & 0\\ 0 & -z & -y & 0\end{pmatrix}.
\] 
It is clear that $D= \ad(aX-yP+zQ)$ and then $D$ is an inner derivation.
\end{proof}

\begin{cor}\label{cor3.9} Any double extension of $\g_4$ by a skew-symmetric derivation is decomposable. 
\end{cor}
\begin{proof}
 
Keep the basis $\{X,P,Q,Z\}$ of $\g_4$ as in the proof of Lemma \ref{lem3.8} and let $D$ be a skew-symmetric derivation of $\g_4$. By the above lemma, the matrix of $D$ is:

\[D = \begin{pmatrix} 0 & 0 & 0 & 0 \\ y & a & 0 & 0 \\ z &
  0 & -a & 0\\ 0 & -z & -y & 0\end{pmatrix}
\] 
where $a,\ y,\ z\in\CC$. 

Let $\bar{\g}=\g_4\oplus \CC e\oplus \CC f$ be the double extension of $\g_4$ by $D$. Then the Lie bracket is defined on $\bar{\g}$ as follows:
\[[e,X]=yP+zQ,\ \ [e,P] = aP-zZ, \ \ [e,Q]=-aQ-yZ, \ \ [X,P]=P+zf,\]
\[[X,Q]=-Q+yf \ \text{and } [P,Q]=Z+af.
\]
It is easy to check that $u=-e+aX-yP+zQ$ is in $\Zs(\bar{\g})$. Moreover, $f\in\Zs(\bar{\g})$ and $B(u,f)=-1$. Therefore $\Zs(\bar{\g})$ is not totally isotropic and then $\bar{\g}$ is decomposable.
\end{proof}

A more general result of Corollary \ref{cor3.9} can be found in \cite{FS96} where any double extension by an inner derivation is decomposable.

\begin{prop}
Let $(\g,B)$ be a solvable quadratic Lie algebra of dimension 6. Assume $\g$ indecomposable. Then there exists a basis $\{Z_1,Z_2,Z_3,X_1,X_2,X_3\}$ of $\g$ such that the bilinear form $B$ is defined by $B(X_i,Z_j)=\delta_{i,j}$, $1\leq i,j\leq 3$, the other are zero and $\g$ is i-isomorphic to each of Lie algebras as follows:

\begin{enumerate}
	\item[(i)] $\g_{6,1}$: $[X_3,Z_2]=Z_1$, $[X_3,X_1]=-X_2$ and $[Z_2,X_1]=Z_3$,
	\item[(ii)] $\g_{6,2}(\lambda)$: $[X_3,Z_1]=Z_1$, $[X_3,Z_2]=\lambda Z_2$, $[X_3,X_1]=-X_1$, $[X_3,X_2]=-\lambda X_2$, $[Z_1,X_1]=Z_3$ and $[Z_2,X_2]=\lambda Z_3$ where $\lambda\in\CC$ and $\lambda \neq 0$. In this case, $\g_{6,2}(\lambda_1)$ and $\g_{6,2}(\lambda_2)$ is i-isomorphic if and only if $\lambda_2=\pm\lambda_1$ or $\lambda_2={\lambda_1}^{-1}$,
	\item[(iii)] $\g_{6,3}$: $[X_3,Z_1]=Z_1$, $[X_3,Z_2]=Z_1+Z_2$, $[X_3,X_1]=-X_1-X_2$, $[X_3,X_2]=-X_2$ and $[Z_1,X_1]=[Z_2,X_1]=[Z_2,X_2]=Z_3$.
\end{enumerate}

\end{prop}
\begin{proof}

Assume $(\g,B)$ an indecomposable solvable quadratic Lie algebra of dimension 6. Then $\g$ is a double extension of a four-dimensional solvable quadratic Lie algebra $\qk$ by a skew-symmetric derivation $C$. By Corollary \ref{cor3.9} and note that $\g$ is indecomposable, $\qk$ must be Abelian. Therefore, $\g$ can be written by $\g = (\CC X_3\oplus\CC Z_3)\oplusp \qk$, where $\qk$ is a four-dimensional quadratic vector space, $C = \ad(X_3)\in\ok(\qk,B_\qk)$, $B(X_3,Z_3) = 1$, $B(X_3,X_3) = B(Z_3,Z_3) = 0$ and $B_\qk = B|_{\qk\times\qk}$ \cite{DPU}. Moreover, the i-isomorphic classification of $\g$ reduces to classify $\OO(\qk)$-orbits of $\ps(\ok(\qk))$, where $\ps(\ok(\qk))$ is denoted by the projective space of $\ok(\qk)$. 

Let $\{Z_1,Z_2,X_1,X_2\}$ be a canonical basis of $\qk$, that means $B_\qk(Z_i,Z_j) = B_\qk(X_i,X_j) = 0$, $B_\qk(Z_i,X_j) = \delta_{ij}$, $1\leq i,j\leq 2$. Since $\g$ is indecomposable then we choose $\OO(\qk)$-orbits of $\ps(\ok(\qk))$ whose representative element $C$ has $\ker(C)\subset\im(C)$ and the matrix of $C$ with respect to the basis $\{Z_1,Z_2,X_1,X_2\}$ is given by one of following matrices:
	
	\begin{enumerate}
		\item[(i)]	$C = \begin{pmatrix} 0 & 1 & 0 & 0 \\ 0 & 0 & 0 & 0 \\ 0 &
  0 & 0 & 0\\ 0 & 0 & -1 & 0\end{pmatrix}$: the nilpotent case,
\item[(ii)] $C(\lambda) = \begin{pmatrix} 1 & 0 & 0 & 0 \\ 0 & \lambda & 0 & 0 \\ 0 &
  0 & -1 & 0\\ 0 & 0 & 0 & -\lambda\end{pmatrix},\ \lambda\neq 0$: the diagonalizable case,
\item[(iii)] $C = \begin{pmatrix} 1 & 1 & 0 & 0 \\ 0 & 1 & 0 & 0 \\ 0 &
  0 & -1 & 0\\ 0 & 0 & -1 & -1\end{pmatrix}$: the invertible case (see \cite{DPU} for more details).
\end{enumerate}

Therefore, we have three not i-isomorphic families of quadratic Lie algebras given as in the proposition which correspond to each of above cases. Note that for those Lie algebras, the i-isomorphic and isomorphic notions are equivalent. For the second family, $\g_{6,2}(\lambda_1)$ is i-isomorphic to $\g_{6,2}(\lambda_2)$ if and only if there exists $\mu\in\CC$ nonzero such that $C(\lambda_1)$ is in the $\OO(\qk)$-adjoint orbit through $\mu C(\lambda_2)$. That happens if and only if $\lambda_1=\pm\lambda_2$ or $\lambda_2={\lambda_1}^{-1}$.

\end{proof}
\bibliographystyle{amsxport}

\end{document}